\documentclass[letterpaper, 10 pt, conference]{ieeeconf}  

\IEEEoverridecommandlockouts 

\usepackage{cite}

\usepackage{xcolor}
\usepackage{epstopdf}
\usepackage{fancyhdr}
\usepackage{subcaption}

\usepackage{amsmath,amsthm,amssymb,mathtools,bbm,dsfont,physics,aligned-overset}
\usepackage[T1]{fontenc}
\usepackage{color}
\usepackage[geometry]{ifsym}
\usepackage{siunitx}
\usepackage{multirow}
\usepackage{placeins}
\usepackage{booktabs}
\usepackage{rotating}
\usepackage{array}
\usepackage{dsfont}
\usepackage{bbm}
\usepackage{bbold}
\usepackage{caption,subcaption}
\usepackage[hyphens]{url}

\usepackage[capitalize]{cleveref}
\usepackage{comment}

\usepackage{epsfig}
\usepackage{epstopdf}

\usepackage{xcolor}

\usepackage[bottom]{footmisc}
\usepackage{graphicx}

\usepackage{mathtools}
\usepackage[acronym]{glossaries}

\newtheorem{theorem}{Theorem}
\newtheorem{lemma}[theorem]{Lemma}

\newtheorem{proposition}[theorem]{Proposition}

\theoremstyle{definition}
\newtheorem{definition}{Definition}
\newtheorem{assumption}{Assumption}
\newtheorem{example}{Example}
\newtheorem*{remark}{Remark}

\usepackage{stmaryrd}

\usepackage{tikz}
\usepackage{lipsum}

\newacronym{acr:cvar}{CVaR}{conditional value-at-risk}
\newacronym{acr:dro}{DRO}{distributionally robust optimization}
\newacronym{acr:kl}{KL}{Kullback-Leibler}
\newacronym{acr:tv}{TV}{Total Variation}
\newacronym{acr:mmd}{MMD}{Maximum Mean Discrepancy}
\newacronym{acr:ols}{OLS}{ordinary least squares}
\newacronym{acr:wrls}{WRLS}{weighted and ridge-regularized least squares}
\newacronym{acr:ot}{OT}{Optimal Transport}
\newcommand{\reals}{\mathbb{R}}
\newcommand{\nonnegativeReals}{\mathbb{R}_{\geq 0}}

\newcommand{\Id}{\mathrm{Id}}
\renewcommand{\d}{\mathrm{d}}
\DeclareMathOperator*{\range}{Range}
\DeclareMathOperator*{\kernel}{Ker}

\newcommand{\eye}{I}
\newcommand{\transpose}[1]{#1^\top}
\newcommand{\T}[1]{\transpose{#1}}
\newcommand{\inverse}[1]{#1^{-1}}
\newcommand{\inv}[1]{\inverse{#1}}
\newcommand{\pseudoinverse}[1]{#1^\dagger}
\newcommand{\pinv}[1]{\pseudoinverse{#1}}
\newcommand{\sample}[2]{\widehat{#1}^{(#2)}}
\newcommand{\entry}[2]{#1_{#2}}

\newcommand{\Q}{\mathbb{Q}}
\renewcommand{\P}{\mathbb{P}}
\newcommand{\empiricalP}[1]{\widehat\P_{#1}}
\newcommand{\empiricalQ}[1]{\widehat\Q_{#1}}
\newcommand{\expectedValue}[2]{\mathbb{E}_{#1}\left[#2\right]}

\newcommand{\diracDelta}[1]{\delta_{#1}}

\newcommand{\tensorProd}[2]{#1\times #2}

\newcommand{\setPlans}[2]{\Gamma(#1,#2)}
\newcommand{\probSpace}[1]{\mathcal{P}(#1)}

\newcommand{\pushforward}[1]{{#1}_{\#}}
\newcommand{\transportCost}[3]{W^{#1}(#2,#3)}

\newcommand{\ball}[3]{\mathbb{B}_{#1}^{#2}(#3)}

\DeclareMathOperator*{\st}{s.t.}
\newcommand{\DS}{\displaystyle}

\title{\LARGE \bf Capture, Propagate, and Control Distributional Uncertainty}
\author{Liviu Aolaritei, Nicolas Lanzetti, and Florian D\"orfler
\thanks{This work was supported by the Swiss National Science Foundation under NCCR Automation, grant agreement 51NF40\_180545.}
\thanks{The authors are with the Automatic Control Laboratory, Department of Electrical Engineering and Information Technology at ETH Z\"urich, Switzerland, {\tt \{aliviu,lnicolas,dorfler\}@ethz.ch}.}%
}

\begin{document}

\maketitle
\thispagestyle{empty}
\pagestyle{empty}

\IEEEpeerreviewmaketitle

\begin{abstract}
We study stochastic dynamical systems in settings where only partial statistical information about the noise is available, e.g., in the form of a limited number of noise realizations. Such systems are particularly challenging to analyze and control, primarily due to an absence of a distributional uncertainty model which: (1) is expressive enough to capture practically relevant scenarios; (2) can be easily propagated through system maps; (3) is invariant under propagation; and (4) allows for computationally tractable control actions. In this paper, we propose to model distributional uncertainty via Optimal Transport ambiguity sets and show that such modeling choice satisfies all of the above requirements. We then specialize our results to stochastic LTI systems, and start by showing that the distributional uncertainty can be efficiently captured, with high probability, within an Optimal Transport ambiguity set on the space of noise trajectories. Then, we show that such ambiguity sets propagate exactly through the system dynamics, giving rise to stochastic tubes that contain, with high probability, all trajectories of the stochastic system. Finally, we show that the control task is very interpretable, unveiling an interesting decomposition between the roles of the feedforward and the feedback control terms. Our results are actionable and successfully applied in stochastic reachability analysis and in trajectory planning under distributional uncertainty.
\end{abstract}

\section{Introduction}
\label{sec:intro}

In the era of data science, it is increasingly common to encounter stochastic (dynamical) systems for which only partial statistical information on the noise is available (e.g., samples). We are therefore confronted with so-called \emph{distributional uncertainty}, whereby not only is the system affected by noise but also the underlying noise probability distribution is unknown and only partially observable. 

In Operation Research and Machine Learning, Wasserstein ambiguity sets have emerged as a prominent model for distributional uncertainty. These are balls in the probability space defined in terms of the Wasserstein distance~\cite{Villani2009a}, a distance between probability distributions and centered at a reference distribution $\empiricalP{}$. Examples of applications are distributionally robust optimization~\cite{Peyman2018,shafieezadeh2023new}, regression and classification~\cite{shafieezadeh2019regularization,ho2020adversarial,aolaritei2022performance}, adversarial training~\cite{wong2019wasserstein}, etc. 

More recently, Wasserstein ambiguity sets, and more generally optimal transport, penetrated the control community, with application in uncertainty quantification in dynamical systems~\cite{aolaritei2022uncertainty, boskos2020data}, model predictive control~\cite{mark2020stochastic,coulson2021distributionally}, distribution steering~\cite{chen2021optimal}, optimal control~\cite{yang2020wasserstein}, multi-agent stochastic optimization~\cite{cherukuri2022data,cherukuri2019cooperative}, linear quadratic differential games~\cite{adu2022optimal}, probability/multi-agent control~\cite{terpin2023dynamic}, and filtering~\cite{singh2020inference,taghvaei2021optimal,krishnan2020probabilistic,shafieezadeh2018wasserstein}, to name a few.
In this paper, we demonstrate that \emph{Optimal Transport (OT) ambiguity sets}, which encompass Wasserstein ambiguity sets, are also easy to propagate. This makes them very natural to model distributional uncertainty in the context of dynamical systems, enjoying the following three desirable features: 
\begin{itemize}
    \item \textit{Expressivity}. They are rich enough to capture relevant stochasticity of real-world systems; see~\cref{sec:OT}.

    \item \textit{Propagation}. They are easily and exactly propagated through linear maps and they are \emph{invariant} under the propagation (i.e., the result of the propagation is itself an OT ambiguity set); see~\cref{sec:propagation}.

    \item \textit{Computation}. They allow for computationally tractable decision-making (e.g., computing worst-case risk over the uncertainty set); see Sections~\ref{sec:OT} and \ref{sec:applications}.
\end{itemize}

More specifically, our contributions are threefold. First, we study the propagation of OT ambiguity sets through linear maps and show that the propagation of an OT ambiguity set is itself an OT ambiguity set or can be tightly upper bounded by an OT ambiguity set.
Second, we instantiate our results in the context of stochastic linear time-invariant systems and unveil a decomposition between the roles of the feedforward and feedback control terms. Finally, we deploy our results 
for stochastic reachability analysis and trajectory planning under distributional uncertainty. Among others, we demonstrate that OT ambiguity sets enable the data-driven robust design of reachability sets and feedforward input trajectories. This way, decisions perform well for the true (but unknown) probability distribution governing the noise even if the decision-maker only disposes of a few samples.

This theory is successfully exploited in the concomitant paper \cite{aolaritei2023wasserstein} to formulate a Wasserstein Tube MPC capable of optimally trading between safety and performance.

\subsection{Mathematical Preliminaries and Notation}
\label{subsec:math}

Throughout the paper, $\mathcal P(\reals^d)$ denotes the space of probability distributions over $\reals^d$. Given $\P, \Q \in \mathcal P(\reals^d)$, we denote by $\P \otimes \Q$ their product distribution and by $\P^{\otimes t}$ the $t$-fold product distribution $\P \otimes \ldots \otimes \P$ with $t$ terms. The delta probability distribution at $x\in\reals^d$ is denoted by $\diracDelta{x}$. We focus on two classes of transformations of probability distributions: \emph{pushforward} via a linear transformation and the \emph{convolution} with a delta distribution. We start with the pushforward:

\begin{definition}
\label{def:pushforwad}
Let $\P\in\probSpace{\reals^d}$ and $A  \in \reals^{m \times d}$. The pushforward of $\P$ via the linear map $x\mapsto Ax$ is denoted by $\pushforward{A}\P$, and is defined by $(\pushforward{A}\P)(\mathcal B)\coloneqq \P(A^{-1}(\mathcal B))$, for all Borel sets $\mathcal B\subset\reals^m$.
\end{definition}

Equivalently, \cref{def:pushforwad} says that if $x\sim\P$, then $\pushforward{A}\P$ is the probability distribution of the random variable $y=Ax$.

\begin{example}
\label{ex:pushforwad:empirical}
Let $\empiricalP{}=\frac{1}{n}\sum_{i=1}^n\delta_{\sample{x}{i}}$ be an empirical distribution. Then, $\pushforward{A}\empiricalP{}=\frac{1}{n}\sum_{i=1}^n\delta_{A \sample{x}{i}}$ is empirical as well, supported on the propagated samples.
\end{example}

Moreover, given $x \sim \P$ on $\reals^d$ and $y \in \reals^d$, $x+y$ is distributed according to the convolution $\delta_y\ast\P$ defined below.

\begin{definition}
\label{def:convolution}
Let $\P\in\probSpace{\reals^d}$ and $y \in \reals^d$. Then, the convolution of $\P$ and $\delta_y$ is denoted by $\delta_y \ast \P$, and is defined by $(\delta_y \ast \P)(\mathcal A) = \P(\mathcal A - y)$, for all Borel sets $\mathcal A\subset\reals^d$.
\end{definition}

Finally, we are interested in probabilistic constraints based on the \emph{conditional value-at-risk (CVaR)}. Given $f:\reals^d \to \reals$ and a random variable $x\sim\Q$ on $\reals^d$, the CVaR of $f(x)$ at probability level $1-\alpha$ is defined as
\begin{align}
\label{eq:CVaR}
    \text{CVaR}_{1-\gamma}^{\Q}(f(x)) = \inf_{\tau \in \reals}\; \tau + \frac{1}{\gamma} \mathbb E_{\mathbb Q} \left[ \max\{0,f(x)-\tau\} \right].
\end{align}


\section{Capture Distributional Uncertainty}
\label{sec:OT}

We start by formalizing the notion of OT ambiguity sets and exposing their expressivity and their geometric, statistical, and computational features and guarantees. 


\subsection{Definition of Optimal Transport ambiguity sets}
\label{subsec:definition:OT}

Consider a non-negative lower semi-continuous function $c:\reals^d\to\nonnegativeReals$ (henceforth, referred to as \emph{transportation cost}) and two probability distributions $\P,\Q\in\probSpace{\reals^d}$. Then, the \emph{OT discrepancy} between $\P$ and $\Q$ is defined by  
\begin{equation}\label{eq:otcost}
    \transportCost{c}{\P}{\Q}
    \coloneqq 
    \inf_{\gamma\in\setPlans{\P}{\Q}}\int_{\reals^d\times \reals^d}c(x_1-x_2)\d\gamma(x_1,x_2),
\end{equation}
where $\setPlans{\P}{\Q}$ is the set of all probability distributions over $\mathbb R^d\times\mathbb R^d$ with marginals $\P$ and $\Q$, often called \emph{transport plans} or \emph{couplings}~\cite{Villani2009a}. 
The semantics are as follows: we seek the minimum cost to transport the probability distribution $\P$ onto the probability distribution $\Q$ when transporting a unit of mass from $x_1$ to $x_2$ costs $c(x_1-x_2)$. 
Intuitively, $\transportCost{c}{\P}{\Q}$ quantifies the discrepancy between $\P$ and $\Q$ 
and it naturally provides us with a definition of ambiguity in the space of probability distributions. In particular, the \emph{OT ambiguity set} of radius $\varepsilon$ centered at $\P$  is defined by 
\begin{equation}\label{eq:otambiguityset}
    \ball{\varepsilon}{c}{\P}
    \coloneqq 
    \{\Q\in\probSpace{\reals^d}: \transportCost{c}{\P}{\Q}\leq\varepsilon\}
    \subset\probSpace{\reals^d}.
\end{equation}
In words, $\ball{\varepsilon}{c}{\P}$ includes all probability distributions onto which $\P$ can be transported with a budget of at most $\varepsilon$.


\subsection{Properties of Optimal Transport ambiguity sets}
\label{subsec:properties:OT}

OT ambiguity sets are attractive to capture distributional uncertainty for various reasons, which we detail next.

\paragraph{Expressivity}
OT ambiguity sets are highly expressive: they contain both continuous and discrete distributions, distributions not concentrated on the support of $\P$, and even distributions whose mass asymptotically escapes to infinity:
\begin{example}
Let $c(x_1-x_2)=|x_1-x_2|^2$ on $\reals$, $\varepsilon>0$, and let $\Q$ be the Gaussian distribution with mean $0$ and variance $\varepsilon$. Then, $\transportCost{c}{\diracDelta{0}}{\Q} = \mathbb E_{\Q}\left[|x|^2  \right]=\varepsilon$. Moreover, $\transportCost{c}{\diracDelta{0}}{\diracDelta{\sqrt{\varepsilon}}}=\varepsilon$, and $\transportCost{c}{\diracDelta{0}}{{\varepsilon}{\frac{1}{n^2}}\diracDelta{n}+(1-{\varepsilon}{\frac{1}{n^2}})\diracDelta{0}}=\varepsilon$.
\end{example}

These properties cease to hold if the discrepancy between probability distributions is measured via the Kullback-Leibler (KL) divergence or Total Variation (TV) distance \cite{gibbs2002choosing}. 

\paragraph{Geometric properties}
OT ambiguity sets encapsulate the geometry that the transportation cost $c$ induces on $\reals^d$; e.g., if $x_0, x_1$ satisfy $c(x_0-x_1)\leq\varepsilon$, then $\delta_{x_1}\in\ball{\varepsilon}{c}{\delta_{x_0}}$. Moreover, \eqref{eq:otcost}  and \eqref{eq:otambiguityset} readily show that ambiguity sets are well-behaved under monotone changes in $c$ and $\varepsilon$:
\begin{lemma}
\label{lemma:ambiguitysetmonotone}
Let $\P\in\probSpace{\reals^d}$, $c, c_1, c_2$ be transportation costs over $\reals^d$, and $\varepsilon,\varepsilon_1, \varepsilon_2>0$. Then,
\begin{itemize}
    \item[(i)] if $\varepsilon_1\leq \varepsilon_2$ then $\ball{\varepsilon_1}{c}{\P}\subseteq\ball{\varepsilon_2}{c}{\P}$;
    \item[(ii)] if $c_1\leq c_2$, then $\ball{\varepsilon}{c_2}{\P}\subseteq\ball{\varepsilon}{c_1}{\P}$.
\end{itemize}
\end{lemma}
In words, an increase in the transportation cost shrinks the ambiguity set, whereas an increase of the radius enlarges it. These simple observations arm practitioners with actionable knobs to control the level of distributional uncertainty.

\paragraph{Statistical properties}
In most applications, probability distributions are not directly observable and must be estimated from data. Specifically, suppose one has access to $n$ i.i.d. samples $\{\sample{x}{i}\}_{i=1}^n$ from $\P$, and constructs the \emph{empirical probability distribution} $\empiricalP{} \coloneqq \frac{1}{n}\sum_{i=1}^n\diracDelta{\widehat{x}^{(i)}}$. 
A straightforward generalization of~\cite[Theorem 2]{Fournier2015} stipulates that if $c(x_1-x_2) \leq \norm{x_1-x_2}^p$ for some $p\geq 1$ and the true distribution $\P$ is light-tailed, then $\P \in \ball{\varepsilon}{c}{\widehat{\P}}$ with high probability, provided that the radius $\varepsilon$ is carefully chosen.

\paragraph{Computational tractability}
For any $\P$-integrable upper semi-continuous function $\ell:\reals^d\to\reals$, the evaluation of the \emph{worst-case risk} over OT ambiguity sets admits a powerful dual reformulation~\cite{Blanchet2019},
\begin{equation*}
    \sup_{\Q\in\ball{\varepsilon}{c}{\P}}
    \expectedValue{\Q}{\ell(x)}
    =
    \inf_{\lambda\geq 0}
    \lambda\varepsilon+\expectedValue{\P}{\sup_{\xi\in\reals^d} \ell(\xi)-\lambda c(\xi-x)},
\end{equation*}
which collapses to computationally tractable finite-dimensional optimization problems for many cases of practical interest~\cite{shafieezadeh2023new}.


\section{Propagate Distributional Uncertainty}
\label{sec:propagation}

In this section, we study how OT ambiguity sets propagate via linear transformations. Before doing so, we show that naive approaches (in particular, propagation of the center only, or propagation based on Lipschitz bounds) fail to effectively capture the propagation of distributional uncertainty.

\subsection{Naive Approaches and their Shortcomings}

Given the OT ambiguity set $\ball{\varepsilon}{c}{\P}$, one might be tempted to approximate the result of the propagation $\pushforward{f}\ball{\varepsilon}{c}{\P}$ by $\ball{\varepsilon}{c}{\pushforward{f}\P}$. This approach suffers from fundamental limitations already in very simple settings, easily resulting in crude overestimation or in catastrophic underestimation of the ambiguity set, as shown in the next example.

\begin{example}
\label{ex:propagation:centernotenough}
Let $c(x_1-x_2)=|x_1-x_2|$ on $\reals$, and $\varepsilon>0$.
\begin{itemize}
    \item Let $A=0$. Then, $\pushforward{A}\ball{\varepsilon}{c}{\P}$ only contains $\diracDelta{0}$, whereas $\ball{\varepsilon}{c}{\pushforward{A}\P}=\ball{\varepsilon}{c}{\diracDelta{0}}$ contains all distributions whose first moment is at most $\varepsilon$. Therefore, $\ball{\varepsilon}{c}{\pushforward{A}\P}$ \emph{overestimates} the true distributional uncertainty.
    
    \item Let $A=2x$, and $\P = \delta_0$. Then, $\ball{\varepsilon}{c}{\pushforward{A}\P}=\ball{\varepsilon}{c}{\diracDelta{0}}$. In~\cref{thm:lin:trans} we show that $\pushforward{A}\ball{\varepsilon}{c}{\P} = \ball{2\varepsilon}{c}{\diracDelta{0}}$. Thus, $\ball{\varepsilon}{c}{\pushforward{A}\P}$ \emph{underestimates} the true uncertainty.
\end{itemize}
\end{example}

Moreover, one might be tempted to bound the propagated distributional uncertainty with the Lipschitz constant $L$ of $A$, i.e., to ``upper bound'' $\pushforward{A}\ball{\varepsilon}{c}{\P}$ with $\ball{L\varepsilon}{c}{\P}$. 
However, this approach suffers from two major limitations. First, the transportation cost $c(x_1-x_2)$ might not be $\norm{x_1-x_2}$, which makes the Lipschitz bound not directly applicable. Indeed, already in~\cref{ex:propagation:centernotenough}, an increase by a factor of 2 in the radius does not alleviate the underestimation of the true ambiguity set: one needs to use $L^2=2^2$ to account for the transportation cost not being $\norm{x_1-x_2}$.
Second, even if the transportation cost is $\norm{x_1-x_2}$, Lipschitz bounds might be overly conservative, as shown next. 

\begin{example}
Let $c(x_1-x_2)=\|x_1-x_2\|$ on $\reals^d$, $\varepsilon>0$, and $A$ be a diagonal matrix with diagonal entries $\{0,0,\ldots,n\}$. Then, $\pushforward{A}\ball{\varepsilon}{c}{\P}$ eliminates all the distributional uncertainty in the first $d-1$ dimensions. However, $\ball{n\varepsilon}{c}{\delta_0}$ contains, among others, all distributions of the form $\delta_0 \otimes \P$ with $\delta_0 \in \mathcal P(\reals)$, and $\P \in \mathcal P(\reals^{d-1})$ satisfying $\mathbb E_{\P}\left[\|x\|  \right]\leq n\varepsilon$.
\end{example}

These shortcomings of naive uncertainty propagation prompt us to study the propagation of OT ambiguity sets.


\subsection{Propagation via Linear Transformations}
\label{subsec:linear:transformations}

We now investigate how OT ambiguity sets are propagated through linear transformations defined by the matrix $A\in\reals^{m\times d}$. To do so, we require the following mild structural assumption on the transportation cost.

\begin{assumption}
\label{assump:transportation:cost}
The transportation cost $c$ is orthomonotone: $c(x_1+x_2) \geq c(x_1)$ for all $x_1,x_2 \in \reals^d$ satisfying $\T{x_1}x_2 = 0$.
\end{assumption}


\begin{remark}
All transportation costs of the form $c(x_1-x_2)=\psi(\norm{x_1-x_2})$, where $\norm{\cdot}$ is the Euclidean norm and $\psi:\nonnegativeReals\to\nonnegativeReals$ is monotone and lower semi-continuous, are orthomonotone: for $x_1,x_2\in\reals^d$ such that $\T{x_1}x_2=0$
\begin{equation*}
    \psi(\norm{x_1+x_2})
    =\psi\left(\sqrt{\norm{x_1}^2+\norm{x_2}^2}\right)
    \geq \psi(\norm{x_1}).
\end{equation*}
\end{remark}

We can now state our main result on the propagation via linear transformations.

\begin{theorem}[Linear transformations]
\label{thm:lin:trans}
Let $\P\in\probSpace{\reals^d}$, and consider a linear transformation defined by a matrix $A \in \reals^{m \times d}$. Moreover, let $c:\reals^d \to \nonnegativeReals$ satisfy~\cref{assump:transportation:cost}. Then, 
\begin{align}
    \label{eq:lin:trans:arbitrary}
        \pushforward{A}{\ball{\varepsilon}{c}{\P}} 
        \subseteq
        \ball{\varepsilon}{c \circ \pinv{A}}{\pushforward{A}\P},
\end{align} 
with $\pinv{A}$ the Moore–Penrose pseudoinverse of $A$. Moreover, if the matrix $A$ is full row-rank, then
\begin{align}
    \label{eq:lin:trans:surjective}
        \pushforward{A}{\ball{\varepsilon}{c}{\P}}
        =
        \ball{\varepsilon}{c \circ \pinv{A}}{\pushforward{A}\P}.
\end{align}
with $\pinv{A} = \T A (A \T A)^{-1}$.
\end{theorem}
\begin{proof}
See Appendix~\ref{appendix:proof:thm1}.
\end{proof}

In words, \cref{thm:lin:trans} asserts that the result of the propagation $\pushforward{A}{\ball{\varepsilon}{c}{\P}}$ is itself an OT ambiguity set, with the same radius $\varepsilon$, propagated center $\pushforward{A}\P$, and an $A$-induced transportation cost $c\circ \pinv{A}$.

\begin{remark}
\label{remark:lin:bounded}
Theorem~\ref{thm:lin:trans} continues to hold if the OT ambiguity set $\ball{\varepsilon}{c}{\P}$ is defined over a subset $\mathcal X \subset \reals^d$, with $\P \in \probSpace{\mathcal X}$, $c:\reals^d \to \nonnegativeReals$, and $A:\mathcal X \to \mathcal Y\coloneqq A \mathcal X$. In that case, the propagated ambiguity set $\ball{\varepsilon}{c \circ \pinv{A}}{\pushforward{A}\P}$ is restricted to all distributions supported on $A \mathcal X$.
\end{remark}

The following example shows that the equality \eqref{eq:lin:trans:surjective} does generally \emph{not} hold for non-surjective linear maps.

\begin{example}
Consider
$
A\coloneqq
\begin{bsmallmatrix}
    1 & 0 \\ 0 & 0
\end{bsmallmatrix}
$
with pseudoinverse $\pinv{A}=A$, the quadratic transportation cost $c(x_1-x_2)=\norm{x_1-x_2}^2$, the probability distribution $\P=\delta_{(0,0)}\in\mathcal{P}(\reals^2)$, and an arbitrary radius $\varepsilon>0$. Let $\Q=\delta_{(0,1)}\in\mathcal{P}(\reals^2)$. Since $(0,1)\not\in\range(A)$, $\Q$ does not belong to $\pushforward{A}\ball{\varepsilon}{c}{\P}$. However, $\transportCost{c\circ\pinv{A}}{\Q}{\P}=\norm{\pinv{A}\begin{bsmallmatrix} 0 \\ 0\end{bsmallmatrix}-\pinv{A}\begin{bsmallmatrix} 0 \\ 1\end{bsmallmatrix}}^2=0$.
Thus, $\Q\in\ball{\varepsilon}{c \circ \pinv{A}}{\pushforward{A}\P}$, and 
$\pushforward{A}\ball{\varepsilon}{c}{\P}\subsetneq\ball{\varepsilon}{c \circ \pinv{A} }{\pushforward{A}\P}$.
\end{example}


\section{Stochastic Linear Control Systems}
\label{sec:LTI}

In this section, we focus on the stochastic linear time-invariant control system
\begin{align}
\label{eq:stoch:dyn:sys}
\begin{split}
    x_{t+1} &= A x_t + B u_t + D w_t\\
    u_t &= K x_t + v_t
\end{split}
\end{align}
and show that the theory of~\cref{sec:propagation} can be efficiently exploited to capture, propagate, and even control distributional uncertainty. 
We assume that the system matrices $A \in \reals^{d \times d}$, $B \in \reals^{d \times m}$, $D \in \reals^{d \times r}$ are known, the initial condition $x_0 \in \reals^d$ is known and deterministic, and the noise sequence $\{w_t\}_{t\in\mathbb N} \subset \reals^r$ is i.i.d. according to an \emph{unknown} light-tailed distribution $\P$ that belongs to the OT ambiguity set $\ball{\varepsilon}{\|\cdot\|_2^2}{\empiricalP{}}$, with reference distribution $\empiricalP{}$ and (translation-invariant and orthomonotone) transportation cost $c(\cdot) = \|\cdot\|_2^2$. 

\begin{remark}
    Modelling the noise using the OT ambiguity set $\ball{\varepsilon}{\|\cdot\|_2^2}{\empiricalP{}}$ allows us to capture many important scenarios. First, we generalize the works which assume the noise to belong to a specific class of distributions (e.g., Gaussian). Second, it allows us to robustly capture the system uncertainty when only partial statistical information about the noise is available (e.g., samples, moments, etc.). In many such cases, $\P \in \ball{\varepsilon}{\|\cdot\|_2^2}{\empiricalP{}}$ can be guaranteed with high probability (see the statistical properties in Section~\ref{subsec:properties:OT}).
\end{remark}


\subsection{Capture Distributional Uncertainty in Linear Systems}
\label{subsec:LTI:capture}

We start by defining, for any $t\in \mathbb N$, the vectors $\mathbf{v}_{[t-1]} = \begin{bmatrix} v_{t-1}^\top & \cdots & v_0^\top \end{bmatrix}^\top$ and $\mathbf{w}_{[t-1]} = \begin{bmatrix}w_{t-1}^\top & \cdots & w_0^\top \end{bmatrix}^\top$, and rewriting the system dynamics~\eqref{eq:stoch:dyn:sys} in the form
\begin{align}
\label{eq:stoch:dyn:sys:0-t}
\begin{split}
    x_{t} &= (A+BK)^t x_0 + \mathbf{B}_{t-1} \mathbf{v}_{[t-1]} + \mathbf{D}_{t-1} \mathbf{w}_{[t-1]}, \\
    \mathbf{B}_{t-1} &= \begin{bmatrix}B & (A+BK) B & \ldots & (A+BK)^{t-1} B \end{bmatrix}, \\ 
    \mathbf{D}_{t-1} &= \begin{bmatrix} D & (A+BK) D & \ldots & (A+BK)^{t-1} D \end{bmatrix}.
\end{split}
\end{align}
\cref{eq:stoch:dyn:sys:0-t} unveils that the distributional uncertainty of the state $x_t$ can be characterized through the pushforward via the matrix $\mathbf{D}_{t-1}$ of the OT ambiguity set associated to $\mathbf{w}_{[t-1]}$.

The following lemma explains how to construct the OT ambiguity set of the noise trajectory $\mathbf{w}_{[t-1]}$, starting from the ambiguity set of $w_0$, i.e., $\ball{\varepsilon}{\|\cdot\|_2^2}{\empiricalP{}}$. Since $\mathbf{w}_{[t-1]}$ is composed of $t$ i.i.d.\ random variables distributed according to $\P$, its distribution is the $t$-fold product distribution $\P^{\otimes t}$. 

\begin{lemma}
\label{lemma:conc:ineq:noise}
Let $\P\in \ball{\varepsilon}{\|\cdot\|_2^2}{\empiricalP{}}$ with probability $1-\delta$, for some $\delta\geq 0$. Then,
\begin{align*}
    \P^{\otimes t} \in \ball{t \varepsilon}{\|\cdot\|_2^2}{\empiricalP{}^{\otimes t}},\;\; \text{with probability $1-\delta$}.
\end{align*}
\end{lemma}
\begin{proof}
See Appendix~\ref{appendix:proof:lemma2}.
\end{proof}

The true power of Lemma~\ref{lemma:conc:ineq:noise} is revealed in data-driven scenarios in the settings of control tasks over a specified prediction horizon $t \in \mathbb N$ (e.g., model predictive control). In such cases, we have access to only $n$ i.i.d.\ noise samples $\{\widehat{w}^{(i)}\}_{i=1}^n$ from $\P$, and we construct the empirical distribution $\empiricalP{} = \frac{1}{n}\sum_{i=1}^n \delta_{\widehat{w}^{(i)}}$. By~\cite[Theorem 2]{Fournier2015}, $\P \in \ball{\varepsilon}{c}{\widehat{\P}}$ with high probability, provided that the radius $\varepsilon$ is in the order of $n^{-1/\max\{2,r\}}$. Then, \cref{lemma:conc:ineq:noise} guarantees that the distribution $\P^{\otimes t}$ of the noise trajectory $\mathbf{w}_{[t-1]}$ belongs, with high-probability, to $\ball{\varepsilon_1}{\|\cdot\|_2^2}{\empiricalP{}^{\otimes t}}$, with radius $\varepsilon_1$ in the order of $t n^{-1/\max\{2,r\}}$. This is preferable over the alternative strategy of working directly with $n$ noise trajectories $\{(\widehat{w}_{t-1}^{(i)},\ldots,\widehat{w}_0^{(i)})\}_{i=1}^n$, and constructing an OT ambiguity set $\ball{\varepsilon_1}{\|\cdot\|_2^2}{\empiricalQ{}}$ around the empirical distribution $\empiricalQ = \frac{1}{n} \sum_{i=1}^n \delta_{(\widehat{w}_{t-1}^{(i)},\ldots,\widehat{w}_0^{(i)})}$ based solely on~\cite[Theorem 2]{Fournier2015}. In that case, the radius $\varepsilon_2$ should be in the order of $n^{-1/\max\{2,t r\}}$. In practical cases, where the dimension of the noise $r$ is low, the \emph{linear} dependence on the horizon in $t n^{-1/t r}$ ensures that the ambiguity radius shrinks much faster with the number of samples  $n$, as opposed to the exponential dependence in $n^{-1/\max\{2,t r\}}$.


\subsection{Propagate and Control Distributional Uncertainty}
\label{subsec:LTI:propagate:control}

We can now study the propagation of the uncertainty from the noise $\mathbf{w}_{[t-1]}$ to the state $x_t$. Importantly, the resulting OT ambiguity set capturing the distributional uncertainty of $x_t$ unveils the role of the two components in the control input $u_t = K x_t + v_t$: the feedforward term $v_t$ controls the \emph{center}, while the feedback gain matrix $K$ controls the \emph{shape and size} of this OT ambiguity set. This is explained in the following proposition and the subsequent discussion.

\begin{proposition}
\label{prop:ambiguity:set:state}
Consider the linear control system~\eqref{eq:stoch:dyn:sys}, with i.i.d.\ noise $\{w_t\}_{t\in\mathbb N}$. Moreover, let $\ball{\varepsilon}{\|\cdot\|_2^2}{\empiricalP{}}$ capture the distributional uncertainty of $w_t$, $\forall t\in \mathbb{N}$ . Then, the distributional uncertainty of $x_t$ is captured by
\begin{align}
\label{eq:ambiguity:x_t}
    \ball{t \varepsilon}{\|\cdot\|_2^2\circ \pinv{\mathbf{D}_{t-1}}}{\delta_{(A+BK)^t x_0 + \mathbf{B}_{t-1} \mathbf{v}_{[t-1]}}\ast(\pushforward{\mathbf{D}_{t-1}}{\empiricalP{}^{\otimes t}})}.
\end{align}
\end{proposition}
\begin{proof}
See Appendix~\ref{appendix:proof:prop1}.
\end{proof}

Recall from~\cref{thm:lin:trans} that the propagation of the distributional uncertainty via $\mathbf{D}_{t-1}$ is \emph{exact} whenever the matrix $\mathbf{D}_{t-1}$ is full row-rank. This trivially holds when $D$ is the identity matrix, and more generally can be guaranteed by an appropriate choice of the feedback gain matrix $K$.
In the following, we inspect the three components of the OT ambiguity set~\eqref{eq:ambiguity:x_t} to shed light on the roles of the feedforward control trajectory $\mathbf{v}_{[t-1]}$ and of the feedback gain matrix $K$.

(1) \textbf{Ambiguity radius $t \varepsilon$}. This quantity grows linearly in the horizon $t$, and can be shrunk only by shrinking $\varepsilon$. This, in turn, requires having access to a higher number of noise samples $\{\widehat{w}^{(i)}\}_{i=1}^n$ (recall that $\varepsilon$ decreases as $n^{-1/\max\{2,r\}}$).

(2) \textbf{Center $\delta_{(A+BK)^t x_0 + \mathbf{B}_{t-1} \mathbf{v}_{[t-1]}}\ast(\pushforward{\mathbf{D}_{t-1}}{\empiricalP{}^{\otimes t}})$}. The center distribution is influenced by both $K$ and $\mathbf{v}_{[t-1]}$. In particular, $K$ determines the shape of the center distribution (through the pushforward via the matrix $\mathbf{D}_{t-1}$); $\mathbf{v}_{[t-1]}$, instead, translates the support of the center distribution in $\reals^d$. This observation becomes more clear in the data-driven scenario: if $\empiricalP{}$ is the empirical distribution supported on the points $\{\widehat{w}^{(i)}\}_{i=1}^n$, then the distribution $\empiricalP{}^{\otimes t}$ is the empirical distribution supported on the noise trajectories $\widehat{\mathbf{w}}_{[t-1]}^{(\mathbf{i})}\coloneqq \begin{bmatrix}(w^{(i_1)})^\top & \cdots & (w^{(i_t)})^\top \end{bmatrix}^\top$, for $\mathbf{i}=[i_1,\ldots,i_t]$, $i_j \in [n]$, $j \in [t]$. In that case, the center distribution becomes the empirical distribution supported on the points
\begin{align*}
    (A+BK)^t x_0 + \mathbf{B}_{t-1} \mathbf{v}_{[t-1]} + \mathbf{D}_{t-1} \widehat{\mathbf{w}}_{[t-1]}^{(\mathbf{i})}.
\end{align*}
Thus, $\mathbf{D}_{t-1}$ (and consequently, $K$) maps the noise trajectory to the point $(A+BK)^t x_0 +\mathbf{D}_{t-1} \widehat{\mathbf{w}}_{[t-1]}^{(\mathbf{i})} \in \reals^d$, and $\mathbf{B}_{t-1} \mathbf{v}_{[t-1]}$ controls this point in $\reals^d$. 

(3) \textbf{Transportation cost $\|\cdot\|_2^2\circ \pinv{\mathbf{D}_{t-1}}$}. This function influences both the shape and size of the OT ambiguity set. For ease of exposition, we assume that $\mathbf{D}_{t-1}$ is full row-rank. Then, $\pinv{\mathbf{D}_{t-1}} = \mathbf{D}_{t-1}^\top (\mathbf{D}_{t-1} \mathbf{D}_{t-1}^\top)^{-1}$. Moreover, if $U \Sigma V^\top$ is the SVD of $\mathbf{D}_{t-1}$, then $V \pinv{\Sigma} U^\top$ is the SVD of $\pinv{\mathbf{D}_{t-1}}$. In particular, the singular values of $\pinv{\mathbf{D}_{t-1}}$ are obtained by taking the inverse of the singular values of $\mathbf{D}_{t-1}$. Consequently, if $\{\sigma_i\}_{i=1}^d$ are the singular values of $\mathbf{D}_{t-1}$ and $\{u_i\}_{i=1}^d$ are the orthonormal columns of $U$, the transportation cost becomes
    \begin{align*}
        \|\pinv{\mathbf{D}_{t-1}}(x_1-x_2)\|_2^2 = \sum_{i=1}^d \frac{1}{\sigma_i^2} \left|u_i^\top (x_1-x_2)\right|^2.
    \end{align*}
    In words, the cost of moving probability mass from the center distribution in the direction $u_i$ costs $\norm{x_1-x_2}^2/\sigma_i^2$ (indeed, $u_i^\top (x_1-x_2)$ is the orthogonal projection of $x_1-x_2$ onto $u_1$). The feedback gain matrix $K$ controls the amount of mass moved in this direction through the singular value $\sigma_i$ of the matrix $\mathbf{D}_{t-1}$. Specifically, the higher the value of $\sigma_i$, the more probability mass is moved in the direction $u_i$. Similarly, the lower the value of $\sigma_i$, the less probability mass is moved in the direction $u_i$. This way, we can precisely control the displacement of probability mass from the center distribution and so the shape and size of the OT ambiguity set~\eqref{eq:ambiguity:x_t}. Alternatively, if only the size of~\eqref{eq:ambiguity:x_t} is of interest, the maximum singular value of $\mathbf{D}_{t-1}$ yields the upper bound
    \begin{align*}
        \ball{t \varepsilon}{\|\cdot\|_2^2\circ \pinv{\mathbf{D}_{t-1}}}{\widetilde \P} \subseteq \ball{t \varepsilon}{\sigma_{\text{min}}(\pinv{\mathbf{D}_{t-1}})^2\|\cdot\|_2^2}{\widetilde \P} = \ball{t \varepsilon \sigma_{\text{max}}^2}{\|\cdot\|_2^2}{\widetilde \P}
    \end{align*}
    with $\widetilde \P\coloneqq\delta_{(A+BK)^t x_0 + \mathbf{B}_{t-1} \mathbf{v}_{[t-1]}}\ast(\pushforward{\mathbf{D}_{t-1}}{\empiricalP{}^{\otimes t}})$ and $\sigma_{\text{max}}\coloneqq\sigma_{\text{max}}(\mathbf{D}_{t-1})$.

Summarizing, the careful inspection of the OT ambiguity set~\eqref{eq:ambiguity:x_t} brings to light the separation of the control tasks carried out by the two components of $u_t$. On the one hand, the feedforward term $v_t$ controls the position in $\reals^d$ of the support of the center distribution (and, with it, the position of the entire OT ambiguity set). On the other hand, the feedback gain matrix $K$ controls the shape and size of the center distribution, as well as the shape and size of the displacement of probability mass from the center distribution (through the transportation cost).



\section{Applications}
\label{sec:applications}

In this section, we apply our theoretic results to stochastic reachability analysis and trajectory planning. In our setting, the decision-maker disposes only of finitely many samples (i.e., noise trajectories) and seeks a deterministic set capturing the state of the system (in reachability analysis) or the cheapest control input to reach a given target (in trajectory planning) which perform well under the true distribution.

\subsection{Preliminaries}
Before diving into our applications, we fix the notation and present a preliminary result in distributionally robust optimization. Henceforth, we assume that the matrix $\mathbf{D}_{t-1}$ is full row-rank and that we have access to $n$ i.i.d.\ noise samples $\{\widehat{w}^{(i)}\}_{i=1}^n$, yielding the $n$ noise sample trajectories 
\begin{align*}
    \widehat{\mathbf{w}}_{[t-1]}^{(i)}:= \begin{bmatrix}(\widehat{w}_{t-1}^{(i)})^\top & \ldots & (\widehat{w}_0^{(i)})^\top \end{bmatrix}^\top,
\end{align*}
for $i \in \{1,\ldots,n\}$. By~\cref{prop:ambiguity:set:state}, the distributional uncertainty of $x_t$ is captured by
\begin{equation*}
    \!\mathbb{S}_t(\!\mathbf{v}_{[t-1]}\!)
    \!\coloneqq\!
    \ball{t \varepsilon}{\|\cdot\|_2^2\circ \pinv{\mathbf{D}_{t-1}}}{\!\delta_{(A+BK)^t x_0 + \mathbf{B}_{t-1} \mathbf{v}_{[t-1]}}\ast(\pushforward{\mathbf{D}_{t-1}}{\empiricalP{}^{\otimes t}})}.
\end{equation*}
In particular, the center of the ambiguity set~\eqref{eq:ambiguity:x_t} is supported on the $N$ controlled state samples
\begin{align*}
    \widehat{x}_t^{(i)} = (A+BK)^t x_0 + \mathbf{B}_{t-1} \mathbf{v}_{[t-1]} + \mathbf{D}_{t-1} \widehat{\mathbf{w}}_{[t-1]}^{(i)}.
\end{align*}
Consider now the polyhedral constraint set
\begin{equation}\label{eq:app:polyhedron}
    \mathcal X \coloneqq 
    \left\{x \in \reals^d:\, \max_{j \in [J]} a_j^\top x + b_j \leq 0,\; J \in \mathbb N\right\},
\end{equation}
and, for some $\gamma \in (0,1)$, we impose the distributionally robust CVaR constraint (with CVaR defined in~\eqref{eq:CVaR})
\begin{align}
\label{eq:CVaR:constraint}
    \sup_{\Q \in \mathbb{S}_t(\mathbf{v}_{[t-1]})} \text{CVaR}_{1-\gamma}^{\Q}\left(\max_{j \in [J]} a_j^\top x_t + b_j\right) \leq 0.
\end{align}
By \cite[Proposition~2.12]{shafieezadeh2023new}, \eqref{eq:CVaR:constraint} can be reformulated as follows: 
\begin{proposition}
\label{prop:DR:CVaR}
Constraint~\eqref{eq:CVaR:constraint} is equivalent to the following set of deterministic constraints, denoted by $\Gamma_t(\mathbf{v}_{[t-1]},A,b)$: 
\begin{align*}
&\quad\forall i \in [N], \forall j \in [J+1]:\\
&\quad\;\begin{cases}
\tau \in \reals, \lambda \in \reals_+, s_i \in \reals&
\\
\lambda \varepsilon_t N + \sum_{i=1}^{N} s_{i} \leq 0 &
\\
\alpha_j^\top \widehat{x}_t^{(i)} + \beta_j(\tau)  + \frac{1}{4\lambda} \alpha_j^\top \left((\pinv{\mathbf{D}_{t-1}})^\top \pinv{\mathbf{D}_{t-1}} \right)^{-1} \alpha_j \leq s_i &
\end{cases}
\end{align*}
with $\alpha_j := a_j/\gamma$ and $\beta_j(\tau)\coloneqq (b_j + \gamma \tau - \tau)/\gamma$, for $j \in [J]$, as well as $\alpha_{J+1}\coloneqq 0$ and $\beta_{J+1}(\tau):=\tau$.
\end{proposition}

\subsection{Distributionally Robust Reachability Analysis}
\label{subsec:reachability}
In the setting of stochastic reachability analysis, for a given feedforward input $\mathbf{v}_{[t-1]}$, we look for the \emph{smallest} (deterministic) set which contains, with high confidence, the state $x_t$ of a stochastic linear time-invariant system at some future time $t$. Specifically, for predefined $a_j\in\reals^n$, $\forall j \in [J]$, we parametrize the set as a polyhedron~\eqref{eq:app:polyhedron} and seek to solve
\begin{equation*}
\begin{array}{cl}
        \max & \sum_{j\in [J]} b_j \\
        \st & \DS \sup_{\Q \in \mathbb{S}_t(\mathbf{v}_{[t-1]})} \text{CVaR}_{1-\gamma}^{\Q}\left(\max_{j \in [J]} a_j^\top x_t + b_j\right) \leq 0.
\end{array}
\end{equation*}
\cref{prop:DR:CVaR} directly gives the convex reformulation
\begin{equation*}
    \begin{array}{cl}
        \max & \DS \sum_{j\in[J]}b_j \\
        \st & \DS b \in \Gamma_t(\mathbf{v}_{[t-1]},A,b).
    \end{array}
\end{equation*}
We evaluate our methodology on the two-dimensional linear system $A=\frac{1}{2}\begin{bsmallmatrix} 1 & -1 \\ 2 & 1\end{bsmallmatrix}$, $B=I$, and $D=0.1 I$, with $K$ being the LQR controller (designed with $Q=R=I$) and $t=10$.
We suppose that the decision-maker has access to 5 noise sample trajectories (the red points in~\cref{fig:rechability}) and that $u_t=0$ (so, $\mathbf{v}_{[t-1]}=0$). We select $\gamma=0.05$ and choose $J=8$ hyperplanes with $a_j=\begin{bmatrix}  i & j \end{bmatrix}^\top$ with $i,j\in\{0,\pm 1\}$ (without the trivial case $a_j=0$). We repeat our experiments for three values of $\varepsilon$. Our results are in~\cref{fig:rechability}. For low $\varepsilon$, the optimal set tightly includes the state resulting from the 5 samples trajectories but performs very poorly on unseen samples (blue crossed~\cref{fig:rechability}). A larger $\varepsilon$, instead, leads to an increasingly larger set, which performs well on test samples, so that $\varepsilon$ arbitrates between performance and robustness.

\begin{figure}[t]
    \centering
    \includegraphics[width=9cm,trim={0 1cm 0 2.8cm},clip]{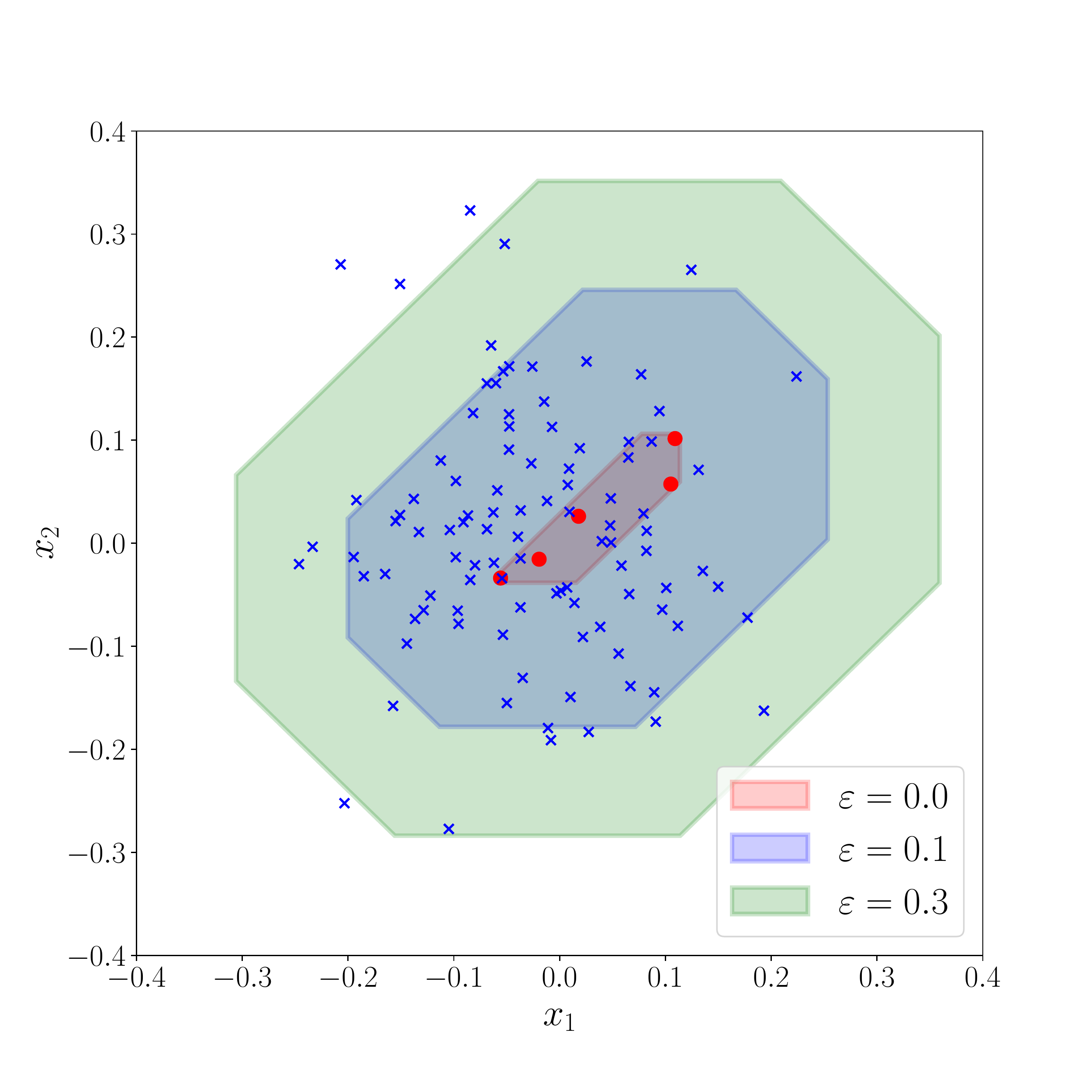}
    \caption{Set containing the state of the system at time $t$ for various radii $\varepsilon$. In red, the training samples (i.e., the samples available to the decision maker); in blue, the testing samples (i.e., samples used to test the decision of the decision-maker). The radius $\varepsilon$ arbitrates performance and robustness: smaller radii lead to smaller sets, which however perform poorly on unseen data points; larger radii result in bigger sets, which well capture unseen samples.}
    \label{fig:rechability}
    \vspace{-.3cm}
\end{figure}

\subsection{Distributionally Robust Trajectory Planning}
\label{subsec:trajectory:planning}

Our second example concerns (distributionally robust) trajectory planning. Given a deterministic initial condition, we look for the cheapest feedforward input $v_t$ steering the system to a given target, expressed in form of a polyhedral set (cf.~\eqref{eq:app:polyhedron}) Accordingly, the trajectory planning problem reads
\begin{equation*}
\begin{array}{cl}
        \min & \sum_{t=0}^{t-1} \norm{v_t}_2^2 \\
        \st & \DS \sup_{\Q \in \mathbb{S}_t(\mathbf{v}_{[t-1]})} \text{CVaR}_{1-\gamma}^{\Q}\left(\max_{j \in [J]} a_j^\top x_t + b_j\right) \leq 0.
\end{array}
\end{equation*}
\cref{prop:DR:CVaR} readily gives the convex reformulation
\begin{align*}
    \begin{array}{cl}
        \min & \DS \|\mathbf{v}_{[t-1]}\|_2^2 \\
        \st & \DS \mathbf{v}_{[t-1]} \in \Gamma_t(\mathbf{v}_{[t-1]},A,b).
    \end{array}
\end{align*}
We apply our methodology to the setting described in~\cref{subsec:reachability} and choose the set $[1,2]\times [1,2]$ as the target (grey in~\cref{fig:trajectory:planning}). As shown in \cref{fig:trajectory:planning}, the feedforward input resulting from $\varepsilon=0$ (red in~\cref{fig:trajectory:planning}) performs well on the 5 sample trajectories, steering them to the boundary of the target set, but yields poor performance on unseen samples. For larger $\varepsilon$ (blue and green in~\cref{fig:trajectory:planning}), instead, the system trajectories are successfully steered to the target set, even for unseen noise realizations, at the price of a slight increase in cost.

\begin{figure}[t]
    \centering
    \includegraphics[width=9cm,trim={0 1cm 0 2.8cm},clip]{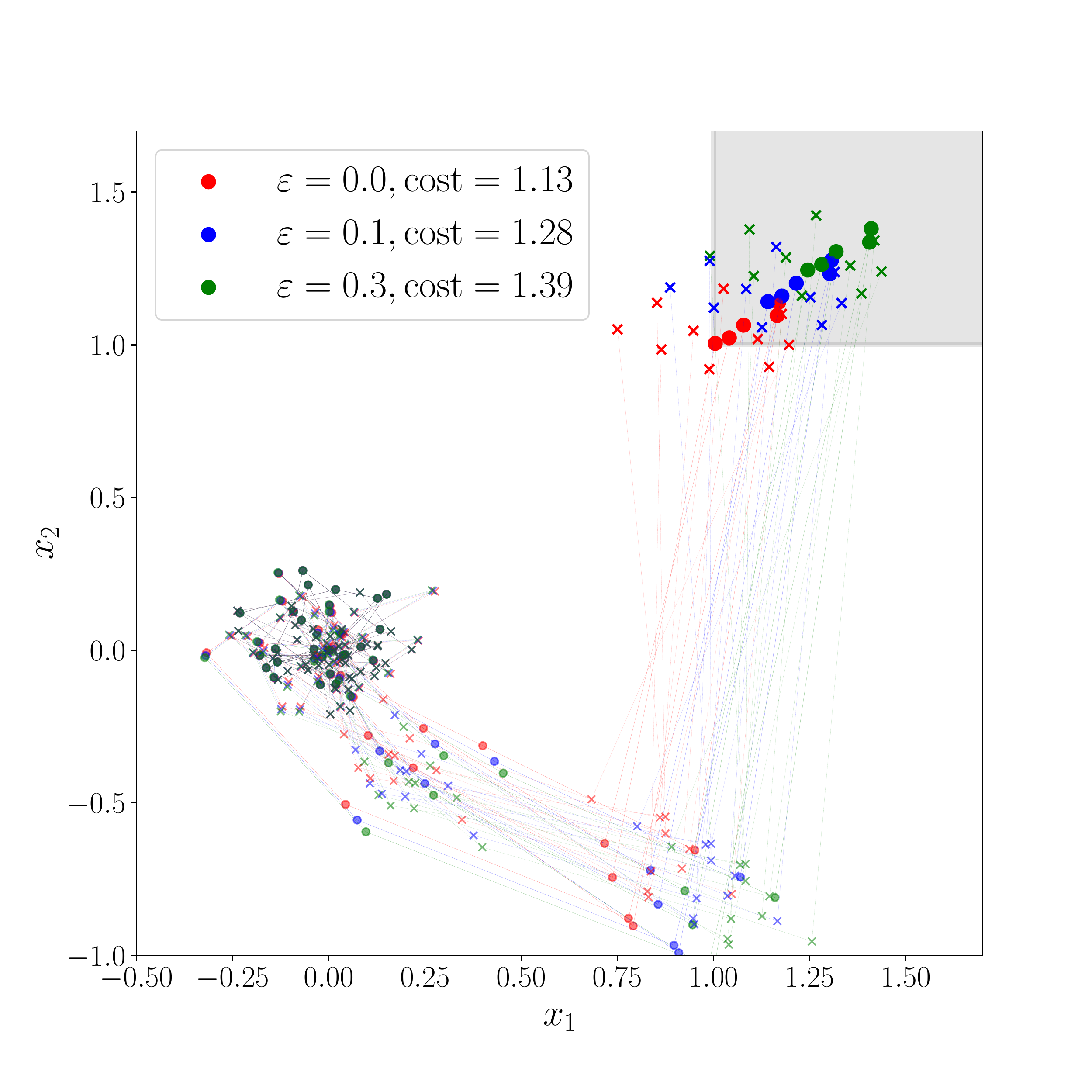}
    \caption{Distributionally robust steering of a system from the origin to a target set (in grey) for various radii $\varepsilon$. Filled circles (with solid lines) are the training samples and crosses (with dotted lines) are testing samples. Small radii lead to trajectories approaching the boundary of the target, which however might not reach the target under different noise realizations (in red); larger radii, instead, push the trajectory to the interior of the target (in green).}
    \label{fig:trajectory:planning}
    \vspace{-.3cm}
\end{figure}


\section{Future work}
\label{sec:conclusions}


A journal version of this paper, which extends these results to nonlinear transformations and to additive and multiplicative stochastic disturbances, is currently in progress.

\bibliographystyle{unsrt}
\bibliography{references}

\section*{Appendix}

\subsection{Technical Preliminaries}

\begin{lemma}[\protect{\cite[Lemma~3.3]{aolaritei2022uncertainty}}]
\label{lemma:stability:plans}
Let $\P, \Q\in\probSpace{\reals^d}$, and consider an arbitrary linear transformation $A:\reals^d \to \reals^m$. Then,
\begin{equation*}
    \pushforward{(\tensorProd{A}{A})}\setPlans{\P}{\Q}
    =\setPlans{\pushforward{A}\P}{\pushforward{A}\Q}.
\end{equation*}
\end{lemma}

\subsection{Proof of Theorem~\ref{thm:lin:trans}}
\label{appendix:proof:thm1}

\textbf{Step~1}: We first show that~\cref{thm:lin:trans} holds for invertible matrices. 

Given $A \in \reals^{d \times d}$ invertible, we start by showing that
\begin{align*}
    \transportCost{c}{\P}{\Q}
    =
    \transportCost{c\circ \inv{A}}{\pushforward{A}\P}{\pushforward{A}\Q}.
\end{align*}
This follows from Lemma~\ref{lemma:stability:plans}, as shown below:
\begin{align*}
    &\transportCost{c}{\P}{\Q}
    =
    \inf_{\gamma\in\setPlans{\P}{\Q}}\int_{\reals^d\times \reals^d}c(x_1-x_2)\d\gamma(x_1,x_2)\\
    &=
    \inf_{\gamma\in\pushforward{(\tensorProd{\inv{A}}{\inv{A}})} \setPlans{\pushforward{f}{\P}}{\pushforward{f}{\Q}}}\int_{\reals^d\times \reals^d}c(x_1-x_2)\d\gamma(x_1,x_2)\\
    &=
    \inf_{\tilde\gamma\in\setPlans{\pushforward{A}{\P}}{\pushforward{A}{\Q}}}\int_{\mathcal Y\times \mathcal Y}c(\inv{A}y_1,\inv{A}y_2)\d\tilde\gamma(y_1,y_2)\\
    &=
    \transportCost{c \circ \inv{A}}{\pushforward{A}\P}{\pushforward{A}\Q},
\end{align*}
where the second equality follows from Lemma~\ref{lemma:stability:plans}, while the other equalities follow the standard properties of pushforward and integral (see Chapter~1 in \cite{Villani2009a}).

We are now ready to prove \eqref{eq:lin:trans:surjective} for invertible matrices. We will first prove the inclusion $\pushforward{A} \ball{\varepsilon}{c}{\P} \subset \mathbb B_\varepsilon^{c \circ \inv{A}}(\pushforward{A} \P)$. Let $\mathbb Q \in \mathbb B_\varepsilon^{c} (\P)$. Then, $\pushforward{A} \mathbb Q \in \mathbb B_\varepsilon^{c \circ \inv{A}}(\pushforward{A} \P)$ follows from the following chain of equivalences
\begin{align*}
    \mathbb Q \in \mathbb B_\varepsilon^{c} (\P) 
    &\iff
    \transportCost{c}{\P}{\Q} \leq \varepsilon \\
    &\iff
    \transportCost{c \circ \inv{A}}{\pushforward{A}\P}{\pushforward{A}\Q}\leq \varepsilon \\
    &\iff 
    \pushforward{A} \mathbb Q \in \mathbb B_\varepsilon^{c \circ \inv{A}}(\pushforward{A} \P).
\end{align*}
We will now prove the converse inclusion, i.e., $\mathbb B_\varepsilon^{c \circ \inv{A}}(\pushforward{A} \P) \subset \pushforward{A} \ball{\varepsilon}{c}{\P}$. This follows from 
\begin{align*}
    \ball{\varepsilon}{c \circ \inv{A}}{\pushforward{A} \P}
    &= 
    \pushforward{(A \circ \inv{A})} \ball{\varepsilon}{c \circ \inv{A}}{\pushforward{A} \P}\\
    &= 
    \pushforward{A} \pushforward{\inv{A}} \ball{\varepsilon}{c \circ \inv{A}}{\pushforward{A} \P}
    \\
    &\subset 
    \pushforward{A} \ball{\varepsilon}{c}{\pushforward{\inv{A}} \pushforward{A} \P}
    \\
    &= 
    \pushforward{A} \ball{\varepsilon}{c}{\P},
\end{align*}
where the inclusion follows using the same reasoning as in the above chain of equivalences. This concludes the proof of \eqref{eq:lin:trans:surjective} for invertible matrices.

\smallskip

\textbf{Step~2}: We will now prove the inclusion \eqref{eq:lin:trans:arbitrary} for arbitrary matrices $A \in \reals^{m \times d}$. 

Let $\Q \in \ball{\varepsilon}{c}{\P}$. Then, $\pushforward{A} \Q \in \mathbb B_\varepsilon^{c \circ \pinv{A}}(\pushforward{A}\P)$ can be shown as follows:
\begin{align*}
    &\inf_{\gamma\in\setPlans{\pushforward{A}\P}{\pushforward{A}\Q}}
    \int_{\reals^m\times \reals^m}c(\pinv{A} y_1-\pinv{A} y_2)\d\gamma(y_1,y_2)
    \\ &=
    \inf_{\gamma\in \pushforward{(A \times A)}\setPlans{\P}{\Q}}
    \int_{\reals^m\times \reals^m}c(\pinv{A} y_1-\pinv{A} y_2)\d\gamma(y_1,y_2)
    \\ &=
    \inf_{\tilde{\gamma}\in \setPlans{\P}{\Q}}\int_{\reals^d\times \reals^d}c(\pinv{A}A x_1-\pinv{A}A x_2)\d\tilde{\gamma}(x_1,x_2)
    \\ &\leq
    \inf_{\tilde{\gamma}\in \setPlans{\P}{\Q}}\int_{\reals^d\times \reals^d}c(x_1- x_2)\d\tilde{\gamma}(x_1,x_2) \leq \varepsilon,
\end{align*}
where the first equality follows from Lemma~\ref{lemma:stability:plans}, and the second to last inequality follows from the orthomonotonicity of $c$ and the fact that $\pinv{A}A$ is the orthogonal projector onto $\kernel(A)^\perp$, i.e.,
\begin{align*}
    c(x_1- x_2) &= c(\pinv{A} A (x_1- x_2) + (\eye_{d\times d} - \pinv{A}A)(x_1- x_2)) \\ &\geq c(\pinv{A}A (x_1- x_2)).
\end{align*}

\smallskip

\textbf{Step~3}: We now focus on the full row-rank case \eqref{eq:lin:trans:surjective}, and show that
\begin{align*}
    \pushforward{A}\ball{\varepsilon}{c}{\P} 
    = 
    \ball{\varepsilon}{c \circ \pinv{A}}{\pushforward{A}\P}.
\end{align*}

Without loss of generality, we can restrict our attention to the case where $A$ is the projection on the first $m$ coordinates, i.e., $A = \pi_{1}:\reals^d\to\reals^m$, with $\pi_{1}(y,z)=y$, for $y \in \reals^m$ and $z \in \reals^{d-m}$. Indeed, assume~\eqref{eq:lin:trans:surjective} holds for the transformation $\pi_{1}$, and let $S\in\reals^{d\times d}$ be the full rank matrix satisfying 
\begin{align*}
    A = \begin{bmatrix}
        \eye_{m \times m} & 0_{m \times (d-m)}
    \end{bmatrix} \begin{bmatrix}
        A \\ B
    \end{bmatrix}
    =
    \pi_1 S,
\end{align*}
for some full row-rank matrix $B \in \reals^{(d-m) \times d}$, with rows linearly independent and orthogonal to the rows of $A$. Then, 
\begin{equation*}
\begin{aligned}
    \pushforward{A}\ball{\varepsilon}{c}{\P}
    =
    \pushforward{(\pi_1S)}\ball{\varepsilon}{c}{\P}
    &=
    \pushforward{\pi_1}
    \ball{\varepsilon}{c \circ S^{-1}}{\pushforward{S}\P}
    \\
    &=
    \ball{\varepsilon}{c \circ S^{-1} \circ \pinv{\pi_1}}{\pushforward{(\pi_1S)}\P}
    \\
    &=
    \ball{\varepsilon}{c \circ \pinv{A}}{\pushforward{A}\P},
\end{aligned}
\end{equation*}
where we have used the fact that $S$ is an invertible transformation and that $\pinv{A} = S^{-1} \pinv{\pi_1}$. In particular, the latter fact follows from the fact that $\pinv{\pi_1} = [\eye_{m \times m} \;\; 0_{m \times (d-m)} \T]$ and
\begin{align*}
    S^{-1}
    &=
    \T S \inv{(S \T S)}\\
    &=
    \begin{bmatrix}\T A & \T B \end{bmatrix} \begin{bmatrix} \inv{(A \T A)} & 0_{m\times (d-m)} \\ 0_{(d-m)\times m} & \inv{(B \T B)}\end{bmatrix} = \begin{bmatrix}\pinv{A} & \pinv{B} \end{bmatrix}.
\end{align*}
In virtue of \eqref{eq:lin:trans:arbitrary} it suffices to prove that
\begin{equation*}
    \ball{\varepsilon}{c \circ \pinv{\pi_1}}{\pushforward{\pi_1}\P}
    \subset
    \pushforward{\pi_1}\ball{\varepsilon}{c}{\P}.
\end{equation*}
Let $\Q\in\ball{\varepsilon}{c\circ\pinv{\pi_1}}{\pushforward{\pi_1}\P}$. Moreover, let $\gamma\in\setPlans{\pushforward{\pi_1}\P}{\Q}$ be the optimal coupling, satisfying 
\begin{equation*}
    \int_{\reals^m\times\reals^m}c(\pinv{\pi_1}y_1-\pinv{\pi_1}y_2)\d\gamma(y_1,y_2)\leq\varepsilon.
\end{equation*}
In the following, in order to avoid confusion, we define by $x$, $x_i$ points in $\reals^d$, by $y$, $y_i$ points in $\reals^m$, and by $z$, $z_i$ points in $\reals^{d-m}$. By the Disintegration Theorem, 
there exists a $({\pi_1}_\# \P)$-almost everywhere uniquely determined family of probability distributions $\{(\P_{y})\}_{y\in\reals^m}$ on $\reals^{d-m}$, such that 
\begin{equation*}
    \d \P(y,z) = \d(\pushforward{\pi_1}\P)(y) \otimes \d\P_{y}(z).
\end{equation*}
Moreover, there exists a $({\pi_1}_\# \P)$-almost everywhere uniquely determined family of probability distributions $\{(\gamma_{y_1})\}_{y_1\in\reals^m}$ on $\reals^{m}$, such that
\begin{equation*}
    \d\gamma(y_1,y_2) = \d(\pushforward{\pi_1}\P)(y_1) \otimes \d\gamma_{y_1}(y_2).
\end{equation*}
Consider now the probability distribution $\bar\Q$ on $\reals^d$:
\begin{equation*}
\begin{aligned}
    \d\bar\Q(y_1,z_1) \coloneqq \int_{\reals^m} \d(\gamma_{y_2} \otimes \P_{y_2})(y_1,z_1) \,\d(\pushforward{\pi_1}\P)(y_2).
\end{aligned} 
\end{equation*}
In the following, we will show that $\pushforward{\pi_1} \bar\Q = \Q$, and that $\bar\Q \in \ball{\varepsilon}{c}{\P}$; this is enough to conclude the proof. For any Borel and bounded test function $\phi:\reals^m \to \reals$, we have
\begin{equation*}
\begin{aligned}
    \int_{\reals^d}
    &\phi(y_2)\d\bar\Q(y_2,z_2)\\
    &=
    \int_{\reals^d \times \reals^m}
    \phi(y_2)\d(\gamma_{y_1} \otimes \P_{y_1})(y_2,z_2) \,\d(\pushforward{\pi_1}\P)(y_1)
    \\ &=
    \int_{\reals^m \times \reals^m}
    \phi(y_2)\d\gamma_{y_1}(y_2) \,\d(\pushforward{\pi_1}\P)(y_1)
    \\ &=
    \int_{\reals^m \times \reals^m}
    \phi(y_2)\d\gamma(y_1,y_2)
    \\ &=
    \int_{\reals^m}
    \phi(y_2)\d\Q(y_2),
\end{aligned}
\end{equation*}
showing that $\pushforward{\pi_1} \bar\Q = \Q$. 

In the rest of the proof, we will show that $\bar\Q \in \ball{\varepsilon}{c}{\P}$. For this, we first define the coupling
\begin{align*}
    \d\bar\gamma&(y_1,z_1,y_2,z_2)
    \coloneqq
    (\pi_{y_1} \times \pi_{z_1} \times \pi_{y_2} \times \pi_{z_2})_\# \\
    &\left(\d{\pi_1}_\# \P(y_1) \otimes \left((\Id \times \Id)_\#\d\P_{y_1}\right)(z_1,z_2) \otimes \d\gamma_{y_1}(y_2)\right).
\end{align*}
By choosing the test function $\phi(y_1,z_1,y_2,z_2)=\norm{z_1-z_2}$, it can be easily seen that $z_1=z_2$ $\bar\gamma$-almost everywhere. Moreover, by construction, $\bar \gamma\in\setPlans{\P}{\bar\Q}$. Indeed, for any Borel and bounded test function $\phi:\reals^d \to \reals$, we have
\begin{equation*}
\begin{aligned}
    \int_{\reals^d\times\reals^d}&\phi(y_1,z_1)\d\bar\gamma(y_1,z_1,y_2,z_2)\\
    &=
    \int_{\reals^d}\phi(y_1,z_1) \d({\pi_1}_\# \P \otimes\P_{y_1})(y_1, z_1)\\
    &=
    \int_{\reals^d}\phi(y_1,z_1)\d\P(y_1,z_1),
\end{aligned}
\end{equation*}
showing that the first marginal of $\bar\gamma$ is $\P$. Moreover, for any Borel and bounded test function $\phi:\reals^d \to \reals$, we have 
\begin{equation*}
\begin{aligned}
    &\int_{\reals^d\times\reals^d}\phi(y_2,z_2)\d\bar\gamma(y_1,z_1,y_2,z_2)\\
    &=
    \int_{\reals^d\times\reals^m}\phi(y_2,z_2)\d(\gamma_{y_1} \otimes \d\P_{y_1})(y_2,z_2)\d(\pushforward{\pi_1}\P)(y_1)
    \\
    &=
    \int_{\reals^d}\phi(y_2,z_2)\int_{\reals^m}\d(\gamma_{y_1} \otimes \d\P_{y_1})(y_2,z_2)\d(\pushforward{\pi_1}\P)(y_1)
    \\
    &=
    \int_{\reals^d}\phi(y_2,z_2)\d\bar\Q(y_2,z_2),
\end{aligned} 
\end{equation*}
showing that the second marginal of $\bar\gamma$ is $\bar \Q$. Finally, $\bar\Q \in \ball{\varepsilon}{c}{\P}$ follows from
\begin{equation*}
\begin{aligned}
    \transportCost{c}{\P}{\bar\Q}
    &\leq
    \int_{\reals^d\times\reals^d}c((y_1,z_1)-(y_2,z_2))\d\bar\gamma(y_1,z_1,y_2,z_2)
    \\
    &=
    \int_{\reals^d\times\reals^d}c((y_1-y_2,0))\d\bar\gamma(y_1,z_1,y_2,z_2)
    \\
    &=
    \int_{\reals^m\times\reals^m}c((y_1-y_2,0))\d(\pushforward{\pi_1}\P \otimes \gamma_{y_1})(y_1,y_2)
    \\
    &=
    \int_{\reals^m\times\reals^m}c(\pinv{\pi_1}y_1-\pinv{\pi_1}y_2)\d\gamma(y_1,y_2)
    \leq 
    \varepsilon.
\end{aligned}
\end{equation*}
This concludes the proof of the inclusion $\ball{\varepsilon}{c \circ \pinv{\pi_1}}{\pushforward{\pi_1}\P} \subset \pushforward{\pi_1}\ball{\varepsilon}{c}{\P}$, and, with it, the proof of Theorem~\ref{thm:lin:trans}.

\subsection{Proof of Lemma~\ref{lemma:conc:ineq:noise}}
\label{appendix:proof:lemma2}

Let $\gamma \in \setPlans{\P}{\empiricalP{}}$ be the optimal coupling associated to the transportation cost $c$, and satisfying
\begin{align*}
    \int_{\reals^d \times \reals^d} \|x_i-y_i\|_2^2 \d\gamma(x_i,y_i) 
    \leq 
    \varepsilon
\end{align*}
with probability $1-\delta$. It is easy to see that the product distribution $\otimes_{i=1}^t \gamma$ belongs to the set of coupling $\setPlans{\P^{\otimes t}}{\empiricalP{}^{\otimes t}}$. Then, we have that the following holds with probability $1-\delta$:
\begin{align*}
    \inf_{\gamma \in \setPlans{\P^{\otimes t}}{\empiricalP{ }^{\otimes t}}} &\int_{\reals^{td} \times \reals^{td}} \|x-y\|_2^2 \d\gamma(x,y) \\
    &\leq 
    \int_{\reals^{td} \times \reals^{td}} \|x-y\|_2^2 \d(\otimes_{i=1}^t \gamma)(x,y)
    \\ &=
    \int_{\reals^d \times \reals^d} \sum_{i=1}^t \|\entry{x}{i}-\entry{y}{i}\|_2^2 \d(\otimes_{i=1}^t \gamma)(x,y)
    \\ &=
    \sum_{i=1}^t\int_{\reals^d \times \reals^d} \|\entry{x}{i}-\entry{y}{i}\|_2^2 \d\gamma(\entry{x}{i},\entry{y}{i})
    \\ &\leq
    t \varepsilon,
\end{align*}
From this we can conclude that $\P^{\otimes t} \in \ball{t \varepsilon}{\|\cdot\|_2^2}{\empiricalP{}^{\otimes t}}$ with probability $1-\delta$. This concludes the proof of Lemma~\ref{lemma:conc:ineq:noise}.

\subsection{Proof of Proposition~\ref{prop:ambiguity:set:state}}
\label{appendix:proof:prop1}

By~\cref{lemma:conc:ineq:noise}, the OT ambiguity set $\ball{t \varepsilon}{\|\cdot\|_2^2}{\empiricalP{}^{\otimes t}}$ captures the distributional uncertainty of the noise trajectory $\mathbf{w}_{[t-1]}$. Moreover, since the transportation cost $\|\cdot\|_2^2$ satisfies~\cref{assump:transportation:cost},~\cref{thm:lin:trans} establishes that the distributional uncertainty of the term $\mathbf{D}_{t-1} \mathbf{w}_{[t-1]}$ is captured by $\ball{t \varepsilon}{\|\cdot\|_2^2 \circ \pinv{\mathbf{D}_{t-1}}}{\pushforward{\mathbf{D}_{t-1}}{\empiricalP{}^{\otimes t}}}$. We now need to consider the sum with the deterministic term $(A+BK)^t x_0 + \mathbf{B}_{t-1} \mathbf{v}_{[t-1]}$. This is a simple translation, corresponding to a convolution at the level of the distributions (see Definition~\ref{def:convolution}). Using \cite[Corollary~3.16]{aolaritei2022uncertainty}, this leads to the following OT ambiguity set
\begin{align*}
    \ball{t \varepsilon}{\|\cdot\|_2^2\circ \pinv{\mathbf{D}_{t-1}}}{\delta_{(A+BK)^t x_0 + \mathbf{B}_{t-1} \mathbf{v}_{[t-1]}}\ast(\pushforward{\mathbf{D}_{t-1}}{\empiricalP{}^{\otimes t}})}.
\end{align*}
This concludes the proof of Proposition~\ref{prop:ambiguity:set:state}.

\end{document}